\documentclass[11pt]{article}
\usepackage{graphicx} 
\usepackage{pslatex}
\usepackage{time}
\usepackage{booktabs}
\usepackage{graphicx}
\graphicspath{ {figs/} }
\usepackage[T1]{fontenc}
\usepackage{amssymb,amsmath,amsfonts}
\usepackage{palatino}
\usepackage{color}
\usepackage{fullpage}
\usepackage{hyperref}
\usepackage{comment}
\usepackage{amsthm}
\usepackage{enumerate}
\usepackage{bigints}
\usepackage{systeme,mathtools}
\usepackage{mdframed}
\usepackage{pdflscape}
\usepackage{rotating}
\usepackage{makecell}
\usepackage{matlab-prettifier}
\usepackage{listings}
\usepackage{setspace}
\DeclareMathAlphabet{\mathpzc}{OT1}{pzc}{m}{it}

\newcommand{\bit}{\begin{itemize}}
\newcommand{\eit}{\end{itemize}}

\newenvironment{dsub}[2]{  \begin{array}{ccccccccccccccc}{#1} \\ {#2}}{\end{array} }
\newcommand\bml{\begin{dsub}}
\newcommand\eml{\end{dsub}}

\newenvironment{mat}{\left(\begin{array}{ccccccccccccccc}}{\end{array}\right)}
\newcommand\bcm{\begin{mat}}
\newcommand\ecm{\end{mat}}

\newenvironment{rmat}{\left(\begin{array}{rrrrrrrrrrrrr}}{\end{array}\right)}
\newcommand\brm{\begin{rmat}}
\newcommand\erm{\end{rmat}}

\usepackage{mathtools}
\usepackage{stmaryrd}
\newtheorem{theorem}{Theorem}
\newtheorem{corollary}{Corollary}

\newtheorem{lemma}{Lemma}

\newcommand{\N}{{\mathbb N}}

\newcommand{\R}{{\mathbb R}}

\definecolor{mve}{rgb}{0.7,0.35,0.15}
\definecolor{brght}{rgb}{0.825,0.2625,0.15}
\definecolor{yello}{rgb}{1,0.925,0.65}
\definecolor{bluu}{rgb}{0.65, 0.95, 1}
\definecolor{bluu2}{rgb}{0.2, 0.5, 0.8}

\date{\vspace{-5ex}}

\usepackage{geometry}
\usepackage{float}
\usepackage{array}
\usepackage{pgf}
\usepackage{pgfpages}

\title{A Sharp Kato-Rosenblum Type Theorem for Unbounded n-Tuples}
\author{Rishabh Bhutani and Dan Virgil Voiculescu}

\begin{document}

\maketitle
\begin{abstract}
    We prove a generalization for commuting n-tuples of unbounded self-adjoint operators and the Lorentz $(n,1)$ ideal, $n \ge 3$, of the Kato-Rosenblum theorem. The result is derived from earlier work for bounded operators $\cite{voiculescu1981some}$. Also, a very weak result for $n=2$ unbounded operators and other additional results are obtained. 
\end{abstract}
\section{Introduction}
Let $\mathcal{H}$ be a complex separable infinite-dimensional Hilbert space. By $\mathcal{C}_n^{-}$ we shall denote the Lorentz $(n,1)$-ideal with norm 
\begin{equation*}
    |K|_{n}^{-} =  \sum_{j} s_j \cdot j^{-1+\frac{1}{n}}. \tag*{(\cite{gohberg1978introduction}, \cite{simon2005trace})}
\end{equation*}
The main aim of this note is to prove the following result:
\begin{theorem} \label{thm1}
    Let $(X_1, \ldots, X_n)$ and $(Y_1, \ldots, Y_n)$ be two n-tuples of possibly unbounded commuting self-adjoint operators such that $X_j - Y_j \in \mathcal{C}_n^-$ for all $1 \leq j \leq n$ and let $E_{ac}(X_1, \ldots, X_n)$ be the projection onto the Lebesgue absolutely continuous subspace of $(X_1, \ldots, X_n)$. Assuming $n \ge 3$, the following strong limit exists: 
    \begin{equation*}
        W = s-\lim_{\xi \in \R^n, |\xi| \to \infty} \exp\left(-i \sum_{1 \leq j \leq n} \xi_j Y_j \right) \exp\left(i \sum_{1 \leq j \leq n} \xi_j X_j \right) E_{ac}(X_1, \ldots, X_n)
    \end{equation*}
    where $\xi=(\xi_1, \ldots, \xi_n)$ and the limit is to the point at $\infty$ in the Alexandroff compactification of $\R^n$. 
    
\end{theorem}
It should be noted that by commuting we mean strongly commuting in the sense of $\cite{schmudgen2012unbounded}$ when referring to tuples of commuting unbounded operators throughout the paper. When restricted to n-tuples of bounded operators, the theorem was proved in $\cite{voiculescu1981some}$ (see Theorems 2.3 and 1.5) in a more general form. The approach in $\cite{voiculescu1981some}$ relies on representations of C*-algebras. Here, to deal with unbounded operators \cite{reed1979iii}, we will apply Theorem 1.5 of $\cite{voiculescu1981some}$ to certain bounded operators associated with unbounded ones. 
\par
Just as the $n=1$ case of the classical Kato-Rosenblum theorem, Theorem \ref{thm1} is sharp because there is a Kuroda-type theorem when the perturbation is not restricted to $\mathcal{C}_n^-$ so that $\mathcal{C}_n^-$ is the largest ideal for which the theorem holds $\cite{bercovici1989analogue}$. 
\par
As already noted in $\cite{voiculescu1981some}$, unlike the case $n=1$, all wave operators are equal for $n \ge 3$ and this extends to generalizations. In the case $n=2$, with the exception of some weak existence results, it is not known whether strong wave operators, if they exist, are all equal (see the recent discussion in $\cite{voiculescu2025perturbations}$). Difficulties with $n=2$ were already noted in $\cite{voigt1977perturbation}$.
\par
Along similar lines to the way Theorem \ref{thm1} is derived, we also derive a very weak result for $n=2$ (Theorem \ref{thm3}). 
\par
In addition to the introduction and references, there are three more sections. In section 2, we have collected various preliminaries in particular about the quasicentral modulus, singular subspace from normed ideal perturbation point of view ($\cite{voiculescu1979some}$, $\cite{voiculescu1981some}$), and Theorems 1.5 and 2.3 of $\cite{voiculescu1981some}$. In section 3, we carry out the proof of Theorem \ref{thm1} by constructing algebras of operators to which Theorem 1.5 of $\cite{voiculescu1981some}$ is applied. Section 4 is about additional results. On one hand is the very weak result for $n=2$ unbounded operators, while on another hand are consequences of the $n \ge 3$ strong result like a two spaces version. 

\section{Preliminaries}
Let $\tau = (T_1 , \ldots , T_n) \in (B(\mathcal{H}))^n$ be an n-tuple of bounded operators on $\mathcal{H}$ and let $(\mathcal{J}, |\cdot|_{\mathcal{J}})$ be a normed ideal of compact operators. By $\mathcal{R}_1^+$ or $\mathcal{R}_1^+(\mathcal{H})$ we will denote the finite-rank positive contractions on $\mathcal{H}$. The $\mathcal{J}$-singular subspace of $\tau$ defined in $\cite{voiculescu1981some}$ can be described by its orthogonal projection $E_{\mathcal{J}}^{(0)}(\tau)$, which is the largest projection $P$ such that there is an increasing sequence $(A_m)_{m \in \N}$ in $\mathcal{R}_1^+$ such that $A_m \uparrow P$ as $m \to \infty$, so that $\max_{1 \leq j \leq n} |[A_m, T_j]|_{\mathcal{J}} \to 0$. The existence of such a largest projection $E_{\mathcal{J}}^{(0)}(\tau)$ is proved in $\cite{voiculescu1981some}$ (Theorem 1.2) where it was also proved that $E_{\mathcal{J}}^{(0)}(\tau)$ is a central projection of the von Neumann algebra of $\tau$ (i.e. $(\tau \ \cup \ \tau^*)''$), which of course is the same as being a central projection of the commutant $(\tau \ \cup \ \tau^*)'$.
\par
Another way of stating the definition of $E_{\mathcal{J}}^{(0)}(\tau)$ is that it is the largest reducing projection $P$ of $\tau$ such that $k_{\mathcal{J}}(\tau|_{P\mathcal{H}})=0$, where $k_{\mathcal{J}}(\tau)$ is the quasicentral modulus ($\cite{voiculescu1979some}$) which is defined as follows:
\begin{equation*}
    \text{the least } C \in [0, \infty] \text{ so that } \exists \hspace{.2em} A_m \in \mathcal{R}_1^+ \text{ such that } A_m \uparrow I \text{ and }  \max_{1 \leq j \leq n} |[A_m, T_j]|_{\mathcal{J}} \to C \text { as } m \to \infty.
\end{equation*}
In the case of an n-tuple $\tau$ of commuting self-adjoint operators and $\mathcal{J}= \mathcal{C}_n^-$, it is proved in $\cite{voiculescu1979some}$ (an immediate consequence of Theorem 4.5) that $E_{\mathcal{C}_n^-}^{(0)}(\tau)$ coincides with the orthogonal projection onto the singular subspace of $\tau$ with respect to Lebesgue measure $E_s(\tau)$. 
\par
In $\cite{voiculescu1979some}$, an n-tuple $\tau$ such that $E_{\mathcal{J}}^{(0)}(\tau)=I$ is called $\mathcal{J}$-well behaved (Definition 2.1). More generally, a vector space of operators with countable basis is called $\mathcal{J}$-well behaved if all finite subsets are $\mathcal{J}$-well behaved (Definition 2.1). 
\par
The orthocomplement of $E_{\mathcal{J}}^{(0)}(\tau)$, that is $I-E_{\mathcal{J}}^{(0)}(\tau)$, will be denoted by $E_{\mathcal{J}}(\tau)$. In the case of an n-tuple of commuting self-adjoint operators and $\mathcal{J} = \mathcal{C}_n^-$, we have $E_{\mathcal{C}_n^-}(\tau)=E_{ac}(\tau)$ where $E_{ac}(\tau)$ is the orthogonal projection onto the Lebesgue absolutely continuous subspace of $\tau$.
\par
In the case of a vector space of operators $\mathcal{X}$ with a countable basis, $E_{\mathcal{J}}^{(0)}(\mathcal{X})$ can be defined as 
\begin{equation*}
    \bigwedge \{E_{\mathcal{J}}^{(0)}(\beta) | \beta \subset \mathcal{X}, \beta \text{ is finite}\}
\end{equation*}
which is the largest reducing projection on which $\mathcal{X}$ is $\mathcal{J}$-well behaved (see the discussion preceeding Theorem 1.2 in $\cite{voiculescu1981some}$). We also consider $E_{\mathcal{J}}(\mathcal{X}) = I-E_{\mathcal{J}}^{(0)}(\mathcal{X})$. Note also that $E_{\mathcal{J}}^{(0)}(\mathbb{C}[\tau, \tau^*])=E_{\mathcal{J}}^{(0)}(\tau)$ and $E_{\mathcal{J}}(\mathbb{C}[\tau, \tau^*])=E_{\mathcal{J}}(\tau)$ where $\mathbb{C}[\tau, \tau^*]$ is the *-algebra generated by $\tau$. Even more, if $\mathcal{X}$ is $\mathcal{J}$-well behaved, then $\mathcal{X}$ can be enlarged using a finite number of operators obtained by analytic functional calculus from operators in $\mathcal{X}$ and will remain $\mathcal{J}$-well behaved. We give a brief proof of this fact: By the definitions, we must show that the quasicentral modulus of any finite subset of the enlarged $\mathcal{X}$, denoted $\tilde{\mathcal{X}}$, remains zero. It suffices to show that given any sequence $A_m \in \mathcal{R}_1^+$ with $A_m \uparrow I$ and $\max_{1 \leq j \leq n} |[A_m, T_j]|_{\mathcal{J}} \to 0$, we have $\max_{1 \leq j \leq n} |[A_m, f(T_j)]|_{\mathcal{J}} \to 0$ where the assignment $f \to f(T_j)$ is the analytic functional calculus. Since the resolvents of $T_j$ are norm-bounded on compact subsets of the resolvent set of $T_j$, we have
\begin{equation*} 
\begin{split}
    \max_{1 \leq j \leq n}|[A_m, f(T_j)]|_{\mathcal{J}}&= \frac{1}{2\pi } \max_{1 \leq j \leq m} \left|\int_{\Gamma} f(\lambda) [A_m,(\lambda-T_j)^{-1}] d\lambda\right|_{\mathcal{J}} \\
    &\leq \frac{1}{2\pi}  \max_{1 \leq j \leq n} \int_{\Gamma} |f(\lambda)| \cdot ||(\lambda-T_j)^{-1}||\cdot |[\lambda -T_j, A_m]|_{\mathcal{J}} \cdot ||(\lambda-T_j)^{-1}|| \cdot |d\lambda| \\
    &\leq C_{\Gamma, f} \cdot \max_{1 \leq j \leq n} |[A_m, T_j]|_{\mathcal{J}} \to 0
\end{split}
\end{equation*}
as $m \to \infty$ where $\Gamma$ is any valid curve enclosing the spectrum of $T_j$, $f$ is analytic on $\Gamma$ and its interior, and $C_{\Gamma, f}$ is a sufficently large constant depending on the choice of $\Gamma$ and $f$. This proves the assertion. One easily deduces from the statement that $E_{\mathcal{J}}^{(0)}(\mathcal{X})$ and $E_{\mathcal{J}}(\mathcal{X})$ do not change under such extensions of $\mathcal{X}$. 
\\
The next lemma is a further characterization of $E_{\mathcal{J}}^{(0)}(\mathcal{X})$. 
\begin{lemma} \label{lemma1} Let $\mathcal{X} \subset B(\mathcal{H})$ be a vector subspace with a countable basis. Then, the following are equivalent:
\begin{enumerate}
    \item[(i)] $\xi \in E_{\mathcal{J}}^{(0)}(\mathcal{X}) \mathcal{H}$
    \item[(ii)] There exists $A_m \in \mathcal{R}_1^+$ such that $|[A_m,  X]|_{\mathcal{J}} \to 0$ for all $X \in \mathcal{X}$ and $||A_m \xi-\xi|| \to 0$ as $m \to \infty$
\end{enumerate}
\end{lemma} 
\begin{proof} 
    \underline{\textbf{(i) $\Rightarrow$ (ii):}} Since $k_{\mathcal{J}}(\beta|_{E_{\mathcal{J}}^{(0)}(\mathcal{H})})=0$ for all finite $\beta \subset \mathcal{X}$, we can find a sequence $A_m \in \mathcal{R}_1^+$ such that 
    \begin{equation*}
        A_m \uparrow E_{\mathcal{J}}^{(0)}(\mathcal{X}) \text{ and } |[A_m, X]|_{\mathcal{J}} \to 0 \text{ for all } X \in \beta  \text{ as } m \to \infty.
    \end{equation*}
    If $(X_k)_{k \in \N}$ be a basis of $\mathcal{X}$, then for each $m \in \N$ consider the finite set $\beta_m = \{X_1,  \ldots , X_m \}$. Since $E_{\mathcal{J}}^{(0)}(\mathcal{X}) = \bigwedge_{m=1}^\infty E_{\mathcal{J}}^{(0)}(\beta_m)$ and $\xi \in E_{\mathcal{J}}^{(0)}(\mathcal{X}) \mathcal{H}$, we have $E_{\mathcal{J}}^{(0)}(\beta_m)\xi = \xi$. For each $\beta_m$, there exists an operator $A_m \in \mathcal{R}_1^+$ so that 
    \begin{equation*}
        ||A_m \xi -\xi||=||A_m \xi - E_{\mathcal{J}}^{(0)}(\beta_m)\xi||< \frac{1}{m} \text{ and } |[A_m, X_j]|_{\mathcal{J}} < \frac{1}{m} \text{ for all } 1 \leq j \leq m.
    \end{equation*}
    By construction, these particular $(A_m)_{m \in \N}$ satisfy (ii). 
    \\
    \\
    \underline{\textbf{(ii) $\Rightarrow$ (i):}} Passing to a subsequence, we may assume that, by the Banach–Alaoglu compactness theorem, $A_m$ converges weakly to an operator $A$ that satisfies $ 0 \leq A \leq I$, $A \xi =\xi$, and $[A,X]=0$ for all $X \in \mathcal{X}$. Then, the orthogonal projection $P$ onto $A \mathcal{H}$ commutes with $\mathcal{X}$ and so we have $$|[PA_nP, X|_{P\mathcal{H}}]|_\mathcal{J} = |(P[A_n, X]P)|_{P\mathcal{H}}|_\mathcal{J} \leq |P[A_n, X]P|_\mathcal{J} \leq |[A_n, X]|_\mathcal{J} \to 0, $$ $PA_nP \xrightarrow{w} PAP$, and Ker$(PAP|_{P \mathcal{H}})=\{0\}$. By Lemma 1.1 of $\cite{voiculescu1981some}$, $k_{\mathcal{J}}(\beta|_{P \mathcal{H}})=0$ for all finite $\beta \subset \mathcal{X}$. Thus, $\mathcal{X}$ is $\mathcal{J}$-well behaved on the reducing subspace $P\mathcal{H}$, which by definition yields $P \leq E_{\mathcal{J}}^{(0)}(\mathcal{X})$. Moreover, $A \xi =\xi$ implies $P \xi = \xi$, and so $P \leq E_{\mathcal{J}}^{(0)}(\mathcal{X})$ then gives (i). 
\end{proof}
We conclude the preliminaries section by recalling a C*-algebraic Kato-Rosenblum type theorem for bounded operators from $\cite{voiculescu1981some}$. 
\begin{theorem}\label{thm2}(\cite{voiculescu1981some}, Theorem 1.5)
    Let $\mathcal{A}$ be a unital separable C* algebra with a dense *-subalgebra $1 \in \mathcal{B} \subset \mathcal{A}$ which is a vector space of countable dimension. Let $\rho_1, \rho_2$ be two non-degenerate *-representations of $\mathcal{A}$ on $\mathcal{H}$ and assume $p>2$ is so that $\rho_1(b)-\rho_2(b) \in \mathcal{C}_p^-$ for all $b \in \mathcal{B}$. Let further $u_n \in \mathcal{B} \ \cap \ Z(\mathcal{A)}$ (where $Z(\mathcal{A})$ is the center of $\mathcal{A}$) be unitary elements and assume that \begin{equation*}
        w- \lim_{n \to \infty}  \rho_1(u_n)E_{\mathcal{C}_p^-}(\rho_1(\mathcal{B}))=0.
    \end{equation*}
    Then, the following strong limit exists:
    \begin{equation*}
        W=s- \lim_{n \to \infty} \rho_2(u_n^*) \rho_1(u_n)E_{\mathcal{C}_p^-}(\rho_1(\mathcal{B})).
    \end{equation*}
    
\end{theorem}

\section{The Extension to Unbounded Operators}
To prove Theorem $\ref{thm1}$, we will need to replace the unbounded operators by associated bounded operators. The next lemma records some standard technical facts we will use. 
\begin{lemma} \label{lemma2}
    Let $(\mathcal{J}, |\cdot|_{\mathcal{J}})$ be a normed ideal and let $X,Y$ be unbounded self-adjoint operators so that $X=Y+K$ where $K \in \mathcal{J}$.
    \begin{enumerate}
        \item[(a)] The resolvents (bounded operators) satisfy $(X-iI)^{-1}-(Y-iI)^{-1} \in \mathcal{J}$
        \item[(b)] We also have $\exp(iX)-\exp(iY) \in \mathcal{J}$.
    \end{enumerate}
\end{lemma}
\begin{proof} \;
    \begin{enumerate}
        \item[(a)] This is immediate from $(X-iI)^{-1}-(Y-iI)^{-1}=-(X-iI)^{-1}K(Y-iI)^{-1}$.
        \item[(b)] We have 
        $$\exp(iX)-\exp(iY)=\exp(itX)\exp(i(1-t)Y)|_{t=0}^1 = i \int_0^1 \exp(itX)K \exp(i(1-t)Y) dt$$ which gives the desired result. 
    \end{enumerate}
\end{proof}
To prove Theorem \ref{thm1}, it suffices under the assumptions of the theorem to show that the following strong limit exists:
\begin{equation*}
    s-\lim_{k \to \infty} \exp\left(-i \sum_{1 \leq j \leq n} \xi_j^{(k)} Y_j \right) \exp\left(i \sum_{1 \leq j \leq n} \xi_j^{(k)} X_j \right) E_{ac}(X_1, \ldots, X_n)
\end{equation*}
where $(\xi^{(k)})_{k \in \N}$ is a sequence in $\R^n$ so that $|\xi^{(k)}| \to \infty$ as $k \to \infty$. 
\par 
This will be achieved as an application of Theorem $\ref{thm2}$ to a C*-algebra $\mathcal{A}$ with dense subalgebra $\mathcal{B}$, unitary elements $(u_k)_{k \in \N}$ and representations $\rho_1, \rho_2$ which we now proceed to describe. The ideal $(\mathcal{J}, |\cdot|_{\mathcal{J}})$ will be $(\mathcal{C}_n^-, |\cdot|_{n}^-)$ where $n \ge 3$. The C* algebra $\mathcal{A}$ will be the separable C*-subalgebra of the C*-algebra $C_b(\R^n)$, bounded continuous functions on $\R^n$, which is generated by $C_0(\R^n)$, the continuous functions on $\R^n$ vanishing at $\infty$ (Alexandroff compactification) and the sequence $u_k \in C_b(\R^n)$ where $k \in \{0\} \cup \N$ given by 
\begin{equation*}
    u_0 \equiv 1 \text{ and } u_k(x_1, \ldots, x_n) = \exp \left(i \sum_{1 \leq j \leq n} \xi_j^{(k)} x_j \right).
\end{equation*}
The representations $\rho_1$, $\rho_2$ arise from $C_b(\R^n)$ functional calculus for $(X_1, \ldots, X_n)$ and $(Y_1, \ldots, Y_n)$, that is for all $f \in \mathcal{A}$, 
\begin{equation*}
    \rho_1(f)= f(X_1, \ldots, X_n) \ \text{ and } \ \rho_2(f)=f(Y_1, \ldots, Y_n).
\end{equation*}
The *-subalgebra $\mathcal{B}$ is the *-subalgebra of $\mathcal{A}$ generated (algebraically) by the functions $(u_k)_{k \in \{0\} \cup\N}$ and the functions $r_j(x_1, \ldots, x_n)= (x_j-i)^{-1}$, $1 \leq j \leq n$. 
\par
The next step in deriving Theorem $\ref{thm1}$ from Theorem $\ref{thm2}$ will be to identify $E^{(0)}_{\mathcal{C}_n^-}(\rho_1(\mathcal{B}))$ with the Lebesgue singular subspace projection of the operators $(X_1 , \ldots , X_n)$. We begin with a general observation about $\mathcal{J}$-singular subspace projections. 
\begin{lemma} \label{lemma3}
    If $\mathcal{X}=\mathcal{X}^* \subset B(\mathcal{H})$ is a vector space with a countable basis and $P_n, P \in (\mathcal{X})'$ are reducing projections such that $P_n \uparrow P$ as $n \to \infty$, then 
    \begin{equation*}
        E_{\mathcal{J}}^{(0)}(\mathcal{X}|_{P_n \mathcal{H}}) \uparrow E_{\mathcal{J}}^{(0)}(\mathcal{X}|_{P \mathcal{H}})
    \end{equation*}
    as $n \to \infty$.
\end{lemma}
\begin{proof}
    Replacing $\mathcal{X}$ by $\mathcal{X}|_{P \mathcal{H}}$ reduces the proof to the case when $P=I$. Since $E_{\mathcal{J}}^{(0)}(\mathcal{X})$ is in the center of $(\mathcal{X})'$, we have that 
    \begin{equation*}
        P_n E_{\mathcal{J}}^{(0)}(\mathcal{X})\mathcal{H}=E_{\mathcal{J}}^{(0)}(\mathcal{X})P_n \mathcal{H} \subset E_{\mathcal{J}}^{(0)}(\mathcal{X}) \mathcal{H} \text{ and } P_n E_{\mathcal{J}}^{(0)}(\mathcal{X}) \uparrow E_{\mathcal{J}}^{(0)} (\mathcal{X}) \ \ (P_nE_{\mathcal{J}}(\mathcal{X}) \text{ is a projection}). 
    \end{equation*}
    On the other hand, 
    \begin{equation*}
        P_n E_{\mathcal{J}}^{(0)}(\mathcal{X}) \mathcal{H} = E_{\mathcal{J}}^{(0)}(\mathcal{X}|_{P_n \mathcal{H}})P_n \mathcal{H}.
    \end{equation*}
     Indeed, using Lemma \ref{lemma1} we easily see that 
     \begin{equation*}
         E_{\mathcal{J}}^{(0)}(\mathcal{X}|_{P_n \mathcal{H}})P_n \mathcal{H} \subset E_{\mathcal{J}}^{(0)}(\mathcal{X}) \mathcal{H} \ \cap \ P_n \mathcal{H} = P_n E_{\mathcal{J}}^{(0)}(\mathcal{X}) \mathcal{H}.
     \end{equation*}
    Therefore, we concentrate on the reverse inclusion. 
    \\
    If $\xi \in E_{\mathcal{J}}^{(0)}(\mathcal{X}) \mathcal{H} \ \cap \ P_n \mathcal{H}$, then Lemma $\ref{lemma1}$ implies there exist $A_m \in \mathcal{R}_1^+(\mathcal{H})$ so that $|[A_m, X]|_{\mathcal{J}} \to 0$ for all $X \in \mathcal{X}$ and $||A_m \xi - \xi|| \to 0$. This also holds for $A_m$ replaced by $P_n A_m P_n$ and $\mathcal{X}$ replaced by $\mathcal{X}|_{P_n \mathcal{H}}$ since 
    \begin{equation*}
        P_n [A_m,X] P_n = [P_nA_mP_n, X] \text{ and } ||P_nA_mP_n \xi- \xi|| = ||P_nA_mP_n \xi - P_n \xi|| \leq ||A_m \xi - \xi|| \to 0.
    \end{equation*}
    Thus $\xi \in E_{\mathcal{J}}^{(0)}(\mathcal{X}|_{P_n \mathcal{H}})P_n \mathcal{H}$. In particular, if $\xi \in E_{\mathcal{J}}^{(0)}(\mathcal{X}) \mathcal{H}$, then $P_n \xi \in E_{\mathcal{J}}^{(0)}(\mathcal{X}|_{P_n \mathcal{H}})$ and $||P_n \xi -\xi|| \to 0$. 
\end{proof}
\begin{lemma} \label{lemma4}
    Let $\mathcal{A, B,} \rho_1$ be the C*-algebra, *-subalgebra, and *-representation respectively defined in this section. Then we have 
    \begin{equation*}
        E_{\mathcal{C}_n^-}(\rho_1(\mathcal{B})) = E_{ac}(X_1, \ldots, X_n).
    \end{equation*}
\end{lemma}
\begin{proof}
    We will use the preceding lemma with $\mathcal{X} = \rho_1(\mathcal{B})$, $\mathcal{J}= \mathcal{C}_n^-$, and the spectral projection $P_m$ of $(X_1, \dots, X_n)$ for $[-m,m]^n \subset \mathbb{R}^n$. Since the $X_j$ restricted to $P_m \mathcal{H}$ are bounded operators for all $1 \leq j \leq n$, by the results of $\cite{voiculescu1979some}$, we have that 
    \begin{equation*}
        E_{\mathcal{C}_n^-}^{(0)}(X_1|_{P_m \mathcal{H}}, \ldots, X_n|_{P_m \mathcal{H}})= E_s(X_1|_{P_m \mathcal{H}}, \ldots, X_n|_{P_m \mathcal{H}})= E_s(X_1, \ldots, X_n)P_m \mathcal{H}.
    \end{equation*}
    On the other hand, the $u_k |_{P_m \mathcal{H}}$ and $(X_j-iI)^{-1}|_{P_m \mathcal{H}}$ are obtained by analytic functional calculus applied to $(X_1|_{P_m \mathcal{H}}, \ldots, X_n|_{P_m \mathcal{H}})$ so that 
    \begin{equation*}
        E^{(0)}_{\mathcal{C}_n^-}(X_1|_{P_m \mathcal{H}}, \ldots, X_n|_{P_m \mathcal{H}})=E^{(0)}_{\mathcal{C}_n^-}(X_1|_{P_m \mathcal{H}}, \ldots, X_n|_{P_m \mathcal{H}}, u_1|_{P_m \mathcal{H}}, \ldots, u_N|_{P_m \mathcal{H}}, (X_1-iI)^{-1}|_{P_m \mathcal{H}}, \ldots , (X_n-iI)^{-1}|_{P_m \mathcal{H}})).
    \end{equation*}
    Using the fact that extending this with polynomials in these commuting operators does not change the $\mathcal{C}_n^-$-singular subspace shows that this also equals $E_{\mathcal{C}_n^-}^{(0)}(\rho_1(\mathcal{B})|_{P_m \mathcal{H}})$. This means 
    that 
    \begin{equation*}
        E_{\mathcal{C}_n^-}^{(0)}(\rho_1(\mathcal{B})|_{P_m \mathcal{H}})= E_s(X_1, \ldots , X_n)P_m \mathcal{H}
    \end{equation*}
    and using Lemma $\ref{lemma3}$ we get that 
    \begin{equation*}
        E_{\mathcal{C}_n^-}^{(0)}(\rho_1(\mathcal{B}))= E_s(X_1, \ldots, X_n).
    \end{equation*}
    Passing to the orthocomplements gives the desired conclusion.
\end{proof}
\noindent
\begin{center}
    \underline{We now summarize the proof of Theorem $\ref{thm1}$:}
\end{center}
We apply Theorem $\ref{thm2}$ to $\mathcal{A, B}, \rho_1, \rho_2, u_k$ introduced in this section. Since $n \ge 3$ in Theorem $\ref{thm1}$, the condition $p>2$ for $\mathcal{C}_p^-$ in Theorem $\ref{thm2}$ is satisfied. Also, $\rho_1(b)-\rho_2(b) \in \mathcal{C}_n^-$ for all $b \in \mathcal{B}$ is a consequence of Lemma $\ref{lemma2}$. The spaces $E_{\mathcal{C}^-_n}(\rho_1)$ and $E_{ac}(X_1, \ldots, X_n)$ have been identified in Lemma $\ref{lemma4}$. The fact that $w-\lim_{n \to \infty} \rho_1(u_n)E_{\mathcal{C}_n^-}(\rho_1(\mathcal{B}))=0$ follows from the well-known fact that $\exp(i \sum_{1 \leq j \leq n)} \xi_j^{(k)}X_j)E_{ac}(X_1 , \ldots , X_n ) \xrightarrow{w} 0$ as $|\xi^{(k)}| \to \infty$ when $k \to \infty$. 

\section{Additional Results}
If $n=2$, we have the following very weak result:
\begin{theorem} \label{thm3}
    Let $(X_1, X_2)$, $(Y_1, Y_2)$ be two pairs of possibly unbounded commuting self-adjoint operators such that $X_j-Y_j \in \mathcal{C}_2^-$ for all $1 \leq j \leq 2$. Assume both pairs have purely Lebesgue absolutely continuous joint spectral measures, that is $E_s(X_1, X_2)=E_s(Y_1, Y_2)=0$. Then, for any sequence $(\eta^{(m)})_{m \in \N}$ in $\R^2$ such that $|\eta^{(m)}| \to \infty$ as $m \to \infty$, there is a subsequence $(\xi^{(k)})_{k \in \N}$ such that 
    \begin{equation*}
        W = w-\lim_{k \to \infty} \exp(-i \sum_{j=1}^2 \xi_j^{(k)} Y_j)\exp(i \sum_{j=1}^2 \xi_j^{(k)} X_j)
    \end{equation*}
    exists and Ker $W = $ Ker $W^* = \{0\}$ and $WX_j =Y_j W$ for $j = 1,2$. 
\end{theorem}
The theorem is derived from a theorem about representations of C*-algebra $\mathcal{A}$ with its dense *-subalgebra $\mathcal{B}$ with countable basis as a vector space:
\begin{theorem}\label{thm4}(\cite{voiculescu1981some}, Theorem 1.4)
    Assume $p>1$ and let $\rho_1, \rho_2$ be non-degenerate *-representations of $\mathcal{A}$ on $\mathcal{H}$ so that $E_{C_p^-}^{(0)}(\rho_1)=E_{C_p^-}^{(0)}(\rho_2)=0$ and $\rho_1(b)-\rho_2(b) \in \mathcal{C}_p^-$ for all $b \in \mathcal{B}$. Assume there exist unitary operators $u_n \in \mathcal{B} \ \cap \ Z(\mathcal{A})$ such that $w-\lim_{n \to \infty} \rho_1(u_n)= w-\lim_{n \to \infty} \rho_2(u_n)=0$ and $\rho_2(u_n^*)\rho_1(u_n) \xrightarrow{w} W$. Then, Ker $W = $ Ker $W^* = \{0\}$ and $\rho_2(b)W=W\rho_1(b)$ for all $b \in \mathcal{B}$. 
\end{theorem}
The idea of deriving Theorem $\ref{thm3}$ using Theorem $\ref{thm4}$ is along the same lines as deriving Theorem $\ref{thm1}$ using Theorem $\ref{thm2}$. In addition, the existence of a subsequence $\xi^{(k)}$ for which the weak limit exists is a consequence of the Banach-Alaoglu compactness theorem.
\par
To conclude this section, let us also note two corollaries of Theorem $\ref{thm1}$ (for the bounded operator case, see Corollary 1.6 and Corollary 1.7 in $\cite{voiculescu1981some}$). 
\begin{corollary} \label{cor1}
    Let $(X_1, \ldots, X_n)$ be an n-tuple of possibly unbounded commuting self-adjoint operators acting on a Hilbert space $\mathcal{H}$ where $n \ge 3$. Suppose $T \in B(\mathcal{H})$ such that $[T, X_j] \in \mathcal{C}_n^-$ for all $ 1 \leq j \leq n$. Then, for any sequence $\xi^{(k)} \in \R^n$ such that $|\xi^{(k)}| \to \infty$, the following strong limit exists:
    \begin{equation*}
        s-\lim_{k \to \infty} \exp\left(-i \sum_{j=1}^n \xi_j^{(k)} X_j\right) T \exp\left(i \sum_{j=1}^n \xi_j^{(k)} X_j\right) E_{ac}(X_1, \ldots, X_n)
    \end{equation*}
\end{corollary}
\begin{corollary} \label{cor2}
    Let $(X_1, \ldots, X_n)$ and $(Y_1, \ldots, Y_n)$ be two n-tuples of possibly unbounded commuting self-adjoint operators acting on a Hilbert spaces $\mathcal{H}_1$ and $\mathcal{H}_2$ respectively where $n \ge 3$. Suppose $J$ is a bounded linear map between $\mathcal{H}_1$ and $\mathcal{H}_2$ such that $Y_jJ-JX_j \in \mathcal{C}_n^-$ for all $1 \leq j \leq n$. Then, for any sequence $\xi^{(k)} \in \R^n$ such that $|\xi^{(k)}| \to \infty$, the following strong limit exists:
    \begin{equation*}
        s-\lim_{k \to \infty} \exp\left(-i \sum_{j=1}^n \xi_j^{(k)} Y_j\right) J \exp\left(i \sum_{j=1}^n \xi_j^{(k)} X_j\right) E_{ac}(X_1, \ldots, X_n)
    \end{equation*}
\end{corollary}
\nocite{*}
\bibliographystyle{abbrv}
\bibliography{refs}
\vspace{2em} 

\noindent
\textsc{Rishabh Bhutani} \\
\textit{Email address:} \texttt{rishabh.bhutani@berkeley.edu} \\
Department of Mathematics, University of California, Berkeley, CA  94720-3840

\vspace{1em}

\noindent
\textsc{Dan Virgil Voiculescu} \\
\textit{Email address:} \texttt{dvv@math.berkeley.edu} \\
Department of Mathematics, University of California, Berkeley, CA 94720-3840
\end{document}